\documentclass{amsart}
\usepackage[a4paper,margin=2.5cm]{geometry}
 
\usepackage[T1]{fontenc}
\usepackage{amssymb, amsmath, amsthm}
\usepackage{enumitem}

\usepackage{url}

\usepackage[backend=bibtex,style=alphabetic,sorting=nyt,isbn=false,url=false,doi=true,maxalphanames=10,minalphanames=4,mincitenames=4,maxcitenames=10,minnames=4,maxnames=10,giveninits=true,maxbibnames=99]{biblatex}

\usepackage{xcolor}
\usepackage[colorlinks=true]{hyperref}

\usepackage{physics}
\usepackage{amsmath}
\usepackage{tikz}
\usepackage{mathdots}
\usepackage{yhmath}
\usepackage{cancel}

\usepackage{siunitx}
\usepackage{array}
\usepackage{multirow}
\usepackage{amssymb}
\usepackage{gensymb}
\usepackage{tabularx}
\usepackage{extarrows}
\usepackage{booktabs}
\usetikzlibrary{patterns}
\usetikzlibrary{shapes}
\usetikzlibrary{decorations.pathreplacing}
\usepackage{relsize}

\usepackage{mathtools}

\DeclareFontFamily{U}{stix2bb}{}
\DeclareFontShape{U}{stix2bb}{m}{n}{<->stix2-mathbb}{}

\numberwithin{equation}{section} 

\theoremstyle{plain}

\newtheorem{prop}[equation]{Proposition}

\newtheorem{cor}[equation]{Corollary}

\def\cxi{\mathbf{i}}
\def\oM{\overline{\mathcal{M}}}

\title{Faber's socle intersection numbers via Gromov--Witten theory of elliptic curve}

\author{Xavier Blot}

\address{X.~B.: Korteweg-de Vriesinstituut voor Wiskunde, Universiteit van Amsterdam, Postbus 94248, 1090GE Amsterdam, Nederland}
\email{x.j.c.v.blot@uva.nl}	

\author{Sergey Shadrin}

\address{S.~S.: Korteweg-de Vriesinstituut voor Wiskunde, Universiteit van Amsterdam, Postbus 94248, 1090GE Amsterdam, Nederland}
\email{s.shadrin@uva.nl}	

\author{Ishan Jaztar Singh}

\address{I.~J.~S.: Dipartimento di Matematica ``Tullio Levi-Civita'', Universit\`a degli studi di Padova,
	Via Trieste 63, 35121 Padova, Italia}
\email{jaztar@math.unipd.it}

\begin{document}

\begin{abstract}
The goal of this very short note is to give a new proof of Faber's formula for the socle intersection numbers in the tautological ring of $\mathcal{M}_g$. This new proof exhibits a new beautiful tautological relation that stems from the recent work of Oberdieck--Pixton on the Gromov--Witten theory of the elliptic curve via a refinement of their argument, and some straightforward computation with the double ramification cycles that enters the recursion relations for the Hamiltonians of the KdV hierarchy.
\end{abstract}

\maketitle

\tableofcontents

\section{Introduction}

Faber's formula~\cite{Faber} for the socle intersection numbers in the tautological ring of $\mathcal{M}_g$ has now several proofs, with quite different geometric ideas behind them. It can be stated in terms of the Deligne--Mumford compactification of the moduli spaces of curves $\oM_{g,n}$, $g\geq 1$, $n\geq 1$, as 
\begin{align} \label{eq:socle}
\int_{\oM_{g,n}} \lambda_g\lambda_{g-1} \prod_{i=1}^n \psi_i^{d_i} = \frac{(-1)^{g-1}B_{2g}(2g-3+n)! }{2^{2g-1}\cdot (2g)!} \prod_{i=1}^n \frac 1{(2d_i-1)!! }, 
\end{align}
for all $d_1,\dots,d_n\geq 0$ such that $d_1+\cdots+d_n = g-2+n$. Here $B_{2g}$ denote the Bernoulli numbers.  Originally it was conjectured in~\cite{Faber} and by now it is a well-established statement with the following five proofs based on a variety of quite different ideas: 
\begin{itemize}
	\item Virasoro constraints for the projective plane, \cite{Getzler-Pand, Givental}.
	\item Mumford's formula and related combinatorics, \cite{Liu-Xu}.
	\item Intersections with double ramification cycles, \cite{Buryak-Sha}.
	\item The 3-spin relations and related combinatorics, \cite{Picton-Thesis}.
	\item The half-spin relations and related combinatorics, \cite{GF1,GF2}.
\end{itemize}
The purpose of this note is to give yet another proof that is based on a new tautological relation that follows from a recent work of Oberdieck--Pixton on the Gromov--Witten theory of the elliptic curve~\cite{pixton_oberdieck_18,oberdieck2023quantumcohomologyhilbertscheme}. 

Although the utility of a new proof of a well-known statement might be questionable, it involves a completely new set of powerful ideas and allows us to exhibit a new tautological relation, which are all to be used in the analysis of the quantum integrable systems of Buryak--Rossi~\cite{BurRossi-Quantum}.

\subsection{Acknowledgments} 
The authors thank Paolo Rossi for useful discussions.

X.B. and S.S. were supported by the Dutch Research Council grant OCENW.M.21.233.  I.J.S. is supported by the University of Padova and the Marie Curie Fellowship (project ID 741896) and is affiliated with the INFN under the national project MMNLP. 

\section{New tautological relation} We consider a subset $N$ of the set of stable graphs corresponding to the strata in $\oM_{g,m}$, $g\geq 1$, $m\geq 1$, which have the form of a necklace:
\begin{itemize}
	\item There are $m$ vertices and to each of them we attach exactly one leaf labeled by $i$, $i=1,\dots,m$.
	\item There are $m$ edges, and they connect the vertices in one cycle of length $m$.
\end{itemize}
There are $(m-1)!/2$ graphs of this shape for $m\geq 2$ (for $m=2$ it is just one graph, but then with the automorphism group of order $2$). The graphs $\Gamma\in N$ are further equipped by the genus function $g\colon V(\Gamma)\to \mathbb{Z}_{\geq 0}$, where $V(\Gamma)=\{v_1,\dots,v_m\}$ is the set of vertices, and we assume that the leaf $i$ is attached to $v_i$, $i=1,\dots,m$. The genus condition reads $\sum_{i=1}^m g(v_i)=g-1$. 

It will be a bit more convenient to introduce an orientation on these necklace graphs and consider the set of wheels $\vec{N}$ (necklaces with an added choice of orientation); then we have exactly $(m-1)!$ wheels. 

\begin{prop} \label{prop:relation}
We have the following relation in $H^*(\oM_{g,m})$: 
	\begin{align} \label{eq:Relation}
		\frac 1{(m-1)!} \!\!\!\!\sum_{\footnotesize\substack{\Gamma \in \vec{N}\\ g\colon V(\Gamma)\to\mathbb{Z}_{\geq 0}}} \!\!\!\! (\mathsf{b}_{\Gamma,g})_*\bigg(\bigotimes_{i=1}^m \mathrm{DR}_{g(v_i)}(0,1,-1) \lambda_{g(v_i)} \bigg) = \frac{2\cdot (2g)!}{(-1)^{g-1}B_{2g}(2g-2+m)!} \lambda_g\lambda_{g-1}\prod_{i=1}^m\psi_i \sum_{i=1}^m \frac{1}{\psi_i}.
	\end{align}
	Here $\mathsf{b}_{\Gamma,g}\colon \prod_{i=1}^m \oM_{g(v_i),3} \to \oM_{g,m}$ is the boundary map, $\mathrm{DR}_{g(v_i)}(0,1,-1)$ is the double ramification cycle on $\oM_{g(v_i),3}$ assigned to $v_i$, where the multiplicity $0$ is assigned to the marked point corresponding to the leaf.    
\end{prop}

\begin{proof} This tautological relation comes from two independent computations of the class $c_{g,m}(p^{\otimes m})\lambda_{g-1}$, $g\geq 1$, $m\geq 1$ in~\cite{oberdieck2023quantumcohomologyhilbertscheme}, where $c_{g,m}(p^{\otimes m})\in H^*(\oM_{g,m})\otimes \mathbb{C}[[q]]$ are the classes of the cohomological field theory associated to the Gromov--Witten theory of the elliptic curve, where all primary fields are the classes of a point in the target curve, and $q$ is the variable that controls the degree of the stable maps.
	
Let $\mathrm{QMod} = \mathbb{C}[G_2, G_4, G_6]$ be the algebra of quasimodular forms, where $G_k(q)$ is the
\(k\)-weighted Eisenstein series is given by
\begin{equation}
	G_k(q) = -\frac{B_k}{2k} + \sum_{n \geq 1}q^n \sum_{d \mid n} d^k, \qquad k=2,4,6,\dots.
\end{equation}
It is graded by nonnegative even integers, $\mathrm{QMod} = \bigoplus_{\ell=0}^\infty \mathrm{QMod}_{2\ell}$, where the grading is given by the weight of the modular forms $\mathrm{wt} (G_k) \coloneq k$, $k=2,4,6$.  Note that for all even \(k \geq 2\), $\mathrm{QMod}_k \ni G_k$. Note also that the operator $ q\partial_q$ acts on $\mathrm{QMod}$ increasing the weight by $2$. 

Using the holomorphic anomaly equation, Oberdieck and Pixton prove \cite[Proposition 6.8]{oberdieck2023quantumcohomologyhilbertscheme} that
\begin{align}\label{eq:FirstFormula}
	c_{g,m}(p^{\otimes m})\lambda_{g-1} =  \frac{2\cdot (2g)!}{(-1)^{g-1}B_{2g}(2g-2+m)!} \lambda_g\lambda_{g-1}\prod_{i=1}^m\psi_i \sum_{i=1}^m \frac{1}{\psi_i} \times (q\partial_q)^{m-1}G_{2g}(q).
\end{align}
On the other hand, in~\cite[Proof of Theorem 5]{pixton_oberdieck_18} they give a formula for $c_{g,m}(p^{\otimes m})$ as a sum over graphs:
\begin{align}
	\label{eqn: cgn as sum of graphs}
	c_{g,m}\left(p^{\otimes m}\right) & =\sum _{\Gamma}\frac{1}{|\mathrm{Aut}(\Gamma)|} \sum_k (\mathsf{b}_\Gamma)_{*} \left(\bigotimes_{i=1}^{m}  \Delta_{g_i}\big((k(h))_{h\in v_i}\big) \right) 
	\\ \notag &
	\qquad \qquad \qquad \qquad \times \prod_{\footnotesize\substack {e=\{h,h'\} \\ e \text{ is a loop}}} 2(-1)^{k(h)} \Big(\tfrac{B_{k(h)+k(h')+2}}{2(k(h)+k(h')+2)}+G_{k(h)+k(h')+2}\Big) 
	\\ \notag & 
	\qquad \qquad \qquad \qquad \times \sum_w \!\!\!\prod_{\footnotesize\substack {e=\{h,h'\} \\ e \text{ is not a loop}}}  \!\!\! \frac{(-1)^{k(h')}w(h)^{k(h)+k(h')+1}}{1-q^{w(h)}}.
\end{align}
Here
\begin{itemize}
	\item \(\Gamma\) is a stable graph representing a boundary stratum in $\oM_{g,n}$ and $\mathsf{b}_\Gamma$ is the corresponding boundary map. We think of edges and vertices as subsets of the set  $H(\Gamma)$ of the  half-edges of $\Gamma$. 
	\item  The $m$ legs of \(\Gamma\) labeled by $1,\dots,m$ are attached to $m$ pairwise different vertices \(v_1, \dots, v_m\) of genera $g_1,\dots,g_m\geq 0$, and there are no further additional vertices. 
	\item Each edge of $\Gamma$ can be included in a cycle in the graph, that is, if we cut any edge, the graph remains connected. Loops are allowed. 
	\item The sum over $k$ is the sum over all possible maps $k\colon H(\Gamma)\to \mathbb{Z}_{\geq 0}$ such that for every loop $e=(h,h')$ the sum $k(h)+k(h')$ is even and $k(h)=0$ if $h$ is a leaf. 
	\item The sum over $w$ is the sum over all possible systems of ``kissing weights'' on the set of half-edges of $\Gamma$, which are the functions  $w\colon \mathrm{H}(\Gamma)\rightarrow \mathbb{Z}$ such that
	\begin{itemize}
		\item $w(h)+w(h')=0$ for every edge $e=\{h,h'\}$;
		\item $\sum_{h\in v}w(h)=0$ for every vertex $v$.
		\item $w(h)=0$ whenever $h$ is a leaf or if it belongs to a loop.
	\end{itemize}
	\item If $e=(h,h')$ is not a loop, then it is assumed that $h$ is attached to a vertex with smaller index than $h'$. This gives an orientation on each edge that we refer to as \emph{index orientation} below. 
	\item The classes $\Delta_{\tilde g}(k_1,\dots,k_{\tilde m})$ are defined in such a way that for any $a_1,\dots,a_{\tilde m}\in \mathbb{Z}$, $a_1+\cdots+a_{\tilde m}=0$, we have 
\begin{align}
	\mathrm{DR}_{\tilde g}(a_1,\dots,a_{\tilde m}) = 
	\sum_{k_1,\dots,k_{\tilde m}\in \mathbb{Z}_{\geq 0}} \Delta_{\tilde g}(k_1,\dots,k_{\tilde m}) \prod_{i=1}^{\tilde m} a_i^{k_i}.
\end{align}	
and the latter expression is a symmetric polynomial in $a_i$'s (such expression exists by~\cite{PixPoly,spelier2024polynomialitydoubleramificationcycle}). 
\end{itemize}

When we multiply this expression by $\lambda_{g-1}$, we get contributions only from the necklace graphs for $m\geq 2$, and the contributions from the graph with one vertex and no edges and the graphs with one vertex and one loop for $m=1$. The class assigned to the vertex $v_i$ is subsequently multiplied by $\lambda_{g_i}$. 

In all cases we use only the following classes 
\begin{align}
\Delta_{g_i}(0,0,2g_i)\lambda_{g_i}=\Delta_{g_i}(0,2g_i,0)\lambda_{g_i}=\frac 12\mathrm{DR}_{g_i}(0,1,-1)\lambda_{g_i}	,
\end{align}
to decorate the vertices (we assume here that the first index is for $k(h)=0$, where $h$ is the leaf attached to $v_i$). Finally, we know that the resulting formula should be a quasimodular form of pure weight $2g-2+2m$ (as prescribed by Equation~\eqref{eq:FirstFormula}), which allows us to drastically simplify the computation of the coefficients of the relevant graphs. To this end it might be seen as a refinement of~\cite[Proof of Lemma 6.6]{oberdieck2023quantumcohomologyhilbertscheme}.

In the case $m=1$, the graph with no edges doesn't contribute anything of weight other than $0$ (in fact, it contributes a constant term that offsets the $B_{2g}/4g$ constant in the second line of~\eqref{eqn: cgn as sum of graphs} in the one-loop graph). Thus, the final result is simply given by
\begin{align}
	c_{g,1}\left(p\right)\lambda_{g-1} = \frac{1}{2} (\mathsf{b}_\Gamma)_* (\mathrm{DR}_{g-1}(0,1,-1)\lambda_{g-1}) \cdot 2G_{2g}(q),
\end{align}
where $\Gamma$ is the one-vertex graph with a single loop (and its group of automorphisms has order $2$). This expression matches exactly the left-hand side of~\eqref{eq:Relation} for $m=1$.

In the case $m\geq 2$, there are $(m-1)!/2$ necklace graphs (for $m=2$ it is just one graph, whose automorphism group has order $2$), but we prefer to consider them as $(m-1)!$ oriented necklace graphs, and then divide the resulting expression by $2$, and this way it works equally well for $m=2$, with an additional choice of the genus labels $g_1,\dots,g_m$ for the vertices. The vertices are decorated by $\mathrm{DR}_{g_i}(0,1,-1)\lambda_{g_i}$, and the kissing weights are fully determined by the choice of one kissing weight $a\in \mathbb{Z}$, say, for the half-edge attached to the vertex $v_1$ that points in the direction prescribed by the orientation of the necklace. 

The coefficient for each necklace is then given by
\begin{align}
	\sum_{a\in \mathbb{Z}_{\not= 0}} a^{2g-2} \left(\frac{a}{1-q^a}\right)^{j_+} \left(\frac{-a}{1-q^{-a}}\right)^{j_-},
\end{align}
where $j_+$ (resp., $j_-$) is the number of edges, whose index orientation agrees (resp., disagrees) with the orientation of the necklace. Note that $j_++j_-=m$ and with our conventions $j_+\geq 1$. The latter coefficient can then be computed as
\begin{align} \label{eq:COefficient-to-compute}
\widetilde{\lim_{p\to 1}}	\left(\sum_{a\in \mathbb{Z}_{\not= 0}} (p\partial_p)^{2g-2} P_{j_+,j_-}(p\partial_p, q\partial_q)  \frac{a}{1-q^a} p^a \right),
\end{align}
where $P_{j_+,j_-}$ is some polynomial, the series expansion is assumed to be taken in the ring $0<|q|<|p|<1$, and the operation $\widetilde{\lim}_{p\to 1}$ takes the limit of the expression that is regularized, if necessary, by removing the principal part of the singularity at $p=1$ (as we see below, the principal part doesn't depend on $q$). 

In order to compute~\eqref{eq:COefficient-to-compute}, we rely on the following observations (see e.g.~\cite{pixton_oberdieck_18})):
\begin{itemize}
	\item  In the ring $0<|q|<|p|<1$ the sum $$\sum_{a\in \mathbb{Z}_{\not= 0}} \frac{a}{1-q^a} p^a ,$$ is equal to the shifted by constant Weierstra\ss{} function $\overline{\wp}_\tau(z)\coloneqq {\wp}_\tau(z) + 2G_2(q)$, where $q=\exp(2\pi\cxi \tau)$ and $p=\exp(2\pi\cxi z)$.
	\item The shifted Weierstra\ss{} function $\overline{\wp}_\tau(z)$ expands at $z=0$ as 
	\begin{align} \label{eq:wp-expansion}
		\frac 1{(2\pi\cxi z)^2} + 2 \sum_{\ell=0}^\infty G_{2\ell+2}(q) \frac{(2\pi\cxi z)^{2\ell}}{(2\ell)!}.		
	\end{align}
	\item The limit $p\to 1$ in the ring $0<|q|<|p|<1$ in~\eqref{eq:COefficient-to-compute} is equal to the limit $z\to 0$ for $0<\mathrm{Im}(z)<\mathrm{Im}(\tau)$ in terms of the shifted Weierstra\ss{} function regularized by removing the principal part at $z=0$. 
	\item Note that the principal part of $(\frac{1}{2\pi\cxi}\partial_z)^{2g-2} P_{j_+,j_-}(\frac{1}{2\pi\cxi}\partial_z, q\partial_q) \overline{\wp}_\tau(z)$ at $z=0$ is independent of $\tau$.
\end{itemize}
Now, consider~\eqref{eq:COefficient-to-compute} in terms of $\overline{\wp}_\tau(z)$. Note that with respect to the weight grading, $\mathrm{wt}(p\partial_p) = 1$ and $\mathrm{wt}(q\partial_q) = 2$. Moreover, for arbitrary $j_+,j_-$, the polynomial 
\begin{align}
P_{j_+,j_-}(p\partial_p,q \partial_q) = \frac{(q\partial_q)^{m-1}}{(m-1)!}  + (\text{lower order terms with respect to the weight grading}). 	
\end{align}
So, since we are only interested in the pure weight grading $2g-2+2m$ and it appears to be the top weight grading in Equation~\eqref{eq:COefficient-to-compute}, we only have to apply $\lim_{z\to 0}$ to 
\begin{align}
 \frac{1}{(m-1)!}\left(q\partial_q\right)^{m-1} \left(\frac{1}{2\pi\cxi}\partial_z\right)^{2g-2} \overline{\wp}_\tau(z),
\end{align}
where the latter expression is regular at $z=0$ and the limit gives $\frac{2}{(m-1)!} (q\partial_q)^{m-1} G_{2g}(q)$. 

Thus, for $m\geq 2$, for the pure weight $2g-2+2m$, we obtain the sum over all oriented necklace graphs, whose vertices are decorated by $\mathrm{DR}_{g_i}(0,1,-1)$, with the coefficient $\frac 12 \cdot \frac{2}{(m-1)!} = \frac{1}{(m-1)!}$. This is exactly the expression that we have on the left-hand side of~\eqref{eq:Relation}. 
\end{proof}

\section{Intersection numbers with the double ramification cycles}

In order to intersect Equation~\eqref{eq:Relation} with $\psi$-classes, we need certain specializations of the following general formula:

\begin{prop} \label{prop:intersection}
For any $a_1,a_2\in\mathbb{Z}$ and for any $g\geq 0$ we have:
	\begin{align} \label{eq:DR-intersection}
		\int_{\oM_{g,3}} \mathrm{DR}_g(-a_1-a_2,a_1,a_2)\lambda_g \psi_1^{g} = \sum_{j=0}^g \frac{(2j-1)!!}{(2g+1)!!(2j)!!} \frac{(a_1+a_2)^{2j}(a_1^2-a_1a_2+a_2^2)^{g-j}}{12^g}. 
	\end{align}
\end{prop}

\begin{proof} Noting that this integral is a (homogenenous) polynomial in $a_1,a_2$~\cite[Lemma 3.2]{Bur}, it is sufficient to compute it for $a_1,a_2>0$. To this end, we use the recursion relation for the $\psi$-class on a double ramification cycle in~\cite[Theorem 4]{BSSZ}, which implies (after we multiply the corresponding expression by $\lambda_g$ and cancel the terms vanishing for the dimensional reasons) that
\begin{align} \label{eq:BasicRelDR}
	& (a_1+a_2)(2g+1) \int_{\oM_{g,3}} \mathrm{DR}_g(-a_1-a_2,a_1,a_2)\lambda_g \psi_1^{g} = (a_1+a_2) \int_{\oM_{g,2}} \mathrm{DR}_{g}(-a_1-a_2,a_1+a_2)\lambda_{g} \psi_1^{g-1} +
	\\ \notag & 2 \int_{\oM_{g-1,3}} \mathrm{DR}_{g-1}(-a_1-a_2,a_1,a_2)\lambda_{g-1} \psi_1^{g-1} \left( a_1 \int_{\oM_{1,2}} \mathrm{DR}_{1}(-a_1,a_1)\lambda_{1} + a_2 \int_{\oM_{1,2}} \mathrm{DR}_{1}(-a_2,a_2)\lambda_{1} \right).
\end{align}
This recursion relation is exactly the same as the one used in~\cite{sha} for the intersection numbers with the cycles of admissible covers. Another reincarnation of this relation is the recursion of Buryak--Rossi for the Hamiltonian densities of the KdV hierarchy in~\cite{BurRossi}.

The integrals $\int_{\oM_{g,2}} \mathrm{DR}_g(-b,b)\lambda_g\psi_1^{g-1}$ can either be computed by the same type of recursion, or, otherwise, using the bamboo formula for $\mathrm{DR}_g(-b,b)\lambda_g$ conjectured in~\cite[Conjecture 2.1]{BHS} and proved in~\cite[Theorem 2.2]{BS-bamboo-proof} it is straightforward to see that
\begin{align}
	\int_{\oM_{g,2}} \mathrm{DR}_g(-b,b)\lambda_g\psi_1^{g-1} = b^{2g} 	\int_{\oM_{g,2}} \psi_1^{3g-1} = \frac{b^{2g}}{24^g g!}.
\end{align}
Thus Equation~\eqref{eq:BasicRelDR} is equivalent to the following recursion relation:
\begin{align}
		& (2g+1) \int_{\oM_{g,3}} \mathrm{DR}_g(-a_1-a_2,a_1,a_2)\lambda_g \psi_1^{g} =  \frac{(a_1+a_2)^{2g}}{24^g g!}+
	\\ \notag & \frac{a_1^2-a_1a_2+a_2^2}{12} \int_{\oM_{g-1,3}} \mathrm{DR}_{g-1}(-a_1-a_2,a_1,a_2)\lambda_{g-1} \psi_1^{g-1} . 
\end{align}
It then follows directly that Equation~\eqref{eq:DR-intersection} satisfies this recursion.
\end{proof}

\section{A new proof of the socle intersection numbers}

A direct corollary of Propositions~\ref{prop:relation} and~\ref{prop:intersection} is the following:

\begin{cor} Equation~\eqref{eq:socle} holds for all $g\geq 1$, $n\geq 1$, $d_1,\dots,d_n\geq 0$ such that $d_1+\cdots+d_n=g-2+n$. 
\end{cor}

\begin{proof} Using the string equation~\cite{witten_91}, we see that it is sufficient to prove~\eqref{eq:socle} for $n=m+1$, $m\geq 1$, with $d_{m+1}=0$ and $d_1,\dots,d_m\geq 1$. Then, using the string equation again, we see that under these assumptions 
\begin{align} \label{eq:socle-1}
	& \frac{2\cdot (2g)!} {(-1)^{g-1}B_{2g}(2g-3+n)! } \int_{\oM_{g,n}} \lambda_g\lambda_{g-1} \prod_{i=1}^{n} \psi_i^{d_i} =
	\\ \notag & \int_{\oM_{g,m}} \bigg(\frac{2\cdot (2g)!} {(-1)^{g-1}B_{2g}(2g-2+m)! } \lambda_g\lambda_{g-1} \prod_{i=1}^m\psi_i \sum_{i=1}^m \frac{1}{\psi_i}\bigg) \prod_{i=1}^m \psi_i^{d_i-1}.
\end{align}
By Proposition~\ref{prop:relation} the latter expression is equal to
\begin{align}
	\frac{1}{(m-1)!} \!\!\!\!\sum_{\footnotesize\substack{\Gamma \in \vec{N}\\ g\colon V(\Gamma)\to\mathbb{Z}_{\geq 0}}} \!\!\!\! \prod_{i=1}^m \int_{\oM_{g(v_i),3}} \mathrm{DR}_{g(v_i)}(0,1,-1)\lambda_{g(v_i)}\psi_1^{d_i-1}. 
\end{align}
For the dimensional reason, the integrals in this expression are non-trivial if and only if $g(v_i) = d_i-1$. Thus the genus function $g$ on the vertices is uniquely determined for every graph $\Gamma$, and there are exactly $(m-1)!$ identical summands in the sum. Thus, the expression is equal to
\begin{align}
\prod_{i=1}^m \int_{\oM_{d_i-1,3}} \mathrm{DR}_{d_i-1}(0,1,-1)\lambda_{d_i-1}\psi_1^{d_i-1}, 
\end{align}
which is equal to 
\begin{align}
\prod_{i=1}^m \frac{1}{(2d_i-1)!! 2^{2d_i-2}} = \frac{1}{2^{2g-2} }\prod_{i=1}^m \frac{1}{(2d_i-1)!!},
\end{align}
 by Proposition~\ref{prop:intersection}.
\end{proof}

\printbibliography

\end{document}